\documentclass[11pt,a4paper,reqno]{amsart}

\usepackage{calc,graphicx,amsfonts,amsthm,amscd,epsfig,psfrag,amsmath,amssymb,enumerate,dsfont}
\usepackage[numeric,initials,nobysame]{amsrefs}

\usepackage[showlabels,sections,floats,textmath,displaymath]{}
\usepackage[colorlinks=true, pdfstartview=FitV, linkcolor=blue, citecolor=blue, urlcolor=blue, pagebackref=false]{hyperref}
\reversemarginpar
\newlength\fullwidth
\numberwithin{equation}{section}

\DeclareMathSymbol{\leqslant}{\mathalpha}{AMSa}{"36} % nicer `smaller or equal'
\DeclareMathSymbol{\geqslant}{\mathalpha}{AMSa}{"3E} % nicer `larger or equal'
\DeclareMathSymbol{\eset}{\mathalpha}{AMSb}{"3F}     % nicer `emptyset'
\renewcommand{\leq}{\;\leqslant\;}                   % redef. of < or =
\renewcommand{\geq}{\;\geqslant\;}                   % redef. of > or =
\def\1{\ifmmode {1\hskip -3pt \rm{I}} \else {\hbox {$1\hskip -3pt \rm{I}$}}\fi}

\newtheorem{Theorem}{Theorem}[section]
\newtheorem{MainTheorem}{Theorem}
\newtheorem{Lemma}[Theorem]{Lemma}

\newtheorem{remark}[Theorem]{Remark}

\newtheorem{definition}[Theorem]{Definition}

\newcommand{\cC}{\ensuremath{\mathcal C}}
\newcommand{\cD}{\ensuremath{\mathcal D}}

\newcommand{\cG}{\ensuremath{\mathcal G}}
\newcommand{\cH}{\ensuremath{\mathcal H}}

\newcommand{\cK}{\ensuremath{\mathcal K}}
\newcommand{\cL}{\ensuremath{\mathcal L}}

\newcommand{\bbN}{{\ensuremath{\mathbb N}} }

\newcommand{\bbP}{{\ensuremath{\mathbb P}} }

\newcommand{\bbR}{{\ensuremath{\mathbb R}} }

\newcommand{\bbT}{{\ensuremath{\mathbb T}} }

\newcommand{\bbZ}{{\ensuremath{\mathbb Z}} }

\newcommand{\var}{\operatorname{Var}}

\let\a=\alpha \let\b=\beta   \let\d=\delta  
 \let\g=\gamma \let\h=\eta      \let\l=\lambda
      \let\o=\omega

\let\D=\Delta   \let\G=\Gamma   
\let\O=\Omega      

\def\\{\hfill\break}

\def\thsp{\thinspace}

\def\tthsp{\kern .083333 em}

\def\?{\mskip -10mu}
\def\indbox#1{\hbox to \parindent{\hfil\ #1\hfil} }

\newfam\msafam
\newfam\msbfam
\newfam\eufmfam
\def\hexnumber#1{%
  \ifcase#1 0\or 1\or 2\or 3\or 4\or 5\or 6\or 7\or 8\or
  9\or A\or B\or C\or D\or E\or F\fi}
\font\tenmsa=msam10
\font\sevenmsa=msam7
\font\fivemsa=msam5
\textfont\msafam=\tenmsa
\scriptfont\msafam=\sevenmsa
\scriptscriptfont\msafam=\fivemsa
\edef\msafamhexnumber{\hexnumber\msafam}%
\mathchardef\restriction"1\msafamhexnumber16
\mathchardef\ssim"0218
\mathchardef\square"0\msafamhexnumber03
\mathchardef\eqd"3\msafamhexnumber2C
\def\QED{\ifhmode\unskip\nobreak\fi\quad
  \ifmmode\square\else$\square$\fi}
\font\tenmsb=msbm10
\font\sevenmsb=msbm7
\font\fivemsb=msbm5
\textfont\msbfam=\tenmsb
\scriptfont\msbfam=\sevenmsb
\scriptscriptfont\msbfam=\fivemsb

\font\teneufm=eufm10
\font\seveneufm=eufm7
\font\fiveeufm=eufm5
\textfont\eufmfam=\teneufm
\scriptfont\eufmfam=\seveneufm
\scriptscriptfont\eufmfam=\fiveeufm

\def\({\left(}
\def\){\right)}
\let\neper=e
\let\ii=i

\def\ie{\hbox{\it i.e.\ }}
\let\id=\identity

\let\sset=\subset

\def\nep#1{ \neper^{#1}}

\def\tc{\thsp | \thsp}
\let\<=\langle
\let\>=\rangle

\def\Var{ \mathop{\rm Var}\nolimits }

\def\gap{\mathop{\rm gap}\nolimits}

\outer\def\nproclaim#1 [#2]#3. #4\par{\medbreak \noindent
   \talato(#2){\bf #1 \Thm[#2]#3.\enspace }%
   {\sl #4\par }\ifdim \lastskip <\medskipamount
   \removelastskip \penalty 55\medskip \fi}
\def\thmm[#1]{#1}
\def\teo[#1]{#1}
\def\sttilde#1{%
\dimen2=\fontdimen5\textfont0
\setbox0=\hbox{$\mathchar"7E$}
\setbox1=\hbox{$\scriptstyle #1$}
\dimen0=\wd0
\dimen1=\wd1
\advance\dimen1 by -\dimen0
\divide\dimen1 by 2
\vbox{\offinterlineskip%
   \moveright\dimen1 \box0 \kern - \dimen2\box1}
}
\begin{document}

\title[Mixing time of a kinetically constrained
spin model on trees at criticality]{Mixing time of a kinetically constrained
spin model on trees: power law scaling at criticality}
\author[N. Cancrini]{N. Cancrini}
\email{nicoletta.cancrini@roma1.infn.it}
\address{DIIIE Univ. L'Aquila, 1-67100 L'Aquila, Italy}

\author[F. Martinelli]{F. Martinelli}
\email{martin@mat.uniroma3.it}
\address{Dip. Matematica, Univ. Roma Tre, Largo S.L.Murialdo 00146, Roma, Italy}

\author[C. Roberto]{C. Roberto}
\email{cyril.roberto@math.cnrs.fr}
\address{Universit\'e Paris Ouest Nanterre La D\'efense - Modal'X, 200 avenue de la R\'epublique 92000 Nanterre, France}

\author[C. Toninelli]{C. Toninelli}
\email{Cristina.Toninelli@lpt.ens.fr}
\address{Laboratoire de Probabilit\'es et Mod\`eles Al\`eatoires
  CNRS-UMR 7599 Universit\'es Paris VI-VII 4, Place Jussieu F-75252 Paris Cedex 05 France}

\thanks{Work supported by  the European Research Council through the
  ``Advanced Grant''  PTRELSS 228032.
The first author was also partially supported by PRIN 2009
  protocollo n.2009TA2595.02 and the fourth author by the French
Ministry of Education through the ANR BLAN07-2184264}

\begin{abstract}
On the rooted $k$-ary tree we consider a $0$-$1$ kinetically constrained spin
model in which the occupancy variable at each node is re-sampled with
rate one from
the Bernoulli(p) measure iff all its children are empty. For this
process the following picture was conjectured to hold. As long as
$p$ is below the percolation threshold $p_c=1/k$ the process is
ergodic with a finite relaxation time while, for $p>p_c$, the process
on the infinite tree is no longer ergodic and 
the relaxation time on a finite
regular sub-tree
becomes exponentially large in the depth of the tree. At the critical
point $p=p_c$ the process on the infinite tree is still ergodic but with
an infinite relaxation time. Moreover, on finite sub-trees, the relaxation
time grows polynomially in the depth of the tree. 

The conjecture was recently proved by the second and forth author
except at criticality. 
Here we analyse  the critical and quasi-critical case and prove for
the relevant time scales: (i)
power law behaviour in the depth of
the tree at $p=p_c$ and (ii) power law scaling in
$(p_c-p)^{-1}$ when $p$ approaches $p_c$ from below. Our results,
which are very close to those obtained recently for the Ising model at
the spin glass critical point, 
represent the first rigorous analysis of a kinetically constrained model
at criticality.
\end{abstract}

\maketitle

\thispagestyle{empty}

\section{Introduction}
%\mnote{citare lavori su albero critico}
On the state space $\{0,1\}^{\bbT^k}$, where $\bbT^k$ is the regular
rooted tree 
with $k\ge 2$ children for each node, we consider a constrained spin
model in which each spin, with rate one and iff all its
children are zero, chooses a new value in
$\{0,1\}$ with probability $1-p$ and $p$ respectively. 
This model belongs to the class of \emph{kinetically constrained spin models} which have been introduced in physics
literature to model liquid/glass transition or, more
generally, glassy dynamics (see
\cites{Ritort-Sollich,GarrahanSollichToninelli} for physical background
and \cite{CMRT} for related mathematical work). 
As for most of
the kinetically constrained models, the Bernoulli(p) product measure
$\mu$ is a reversible measure for the process. 

When $k=1$ the model
coincides with the well known East model \cite{JACKLE} (see also
\cites{AD,CMRT,CMST,FMRT,FMRT3} for rigorous analysis). As soon as
$k\ge 2$, the model shares some of the key features of another well known
kinetically constrained system, namely the North East model \cite{LK,CMRT}. More
specifically, since above the critical density $p_c=1/k$ the occupied
vertices begin to percolate (under the reversible measure $\mu$),
blocked clusters appear and time ergodicity is lost. It is
therefore particularly interesting to study the relaxation to
equilibrium in \textit{e.g.}\ finite sub-trees of $\bbT^k$, when the density $p$
is below, equal or above the critical density $p_c=1/k$.    

In
\cite{MT} it was recently proved that, as long as $p<p_c$,  the process on the infinite
tree is exponentially
ergodic with a finite relaxation time
$T_{\rm rel}$. Under the same assumption, on a finite tree with suitable boundary conditions on the leaves
the mixing time was also shown to be linear in the depth of the tree. When instead $p>p_c$
the ergodicity on the infinite tree is lost and both the
relaxation and the mixing
times for finite trees diverge exponentially fast in the depth of the tree.

In this paper we tackle for the first time the critical case $p=p_c$.
Our main results, answering a question of Aldous-Diaconis \cite{AD},
can be formulated as follows. 
\begin{itemize}
  \item \emph{Critical case}. Assume $p=p_c$ and let $\bbT$ be a
 finite $k$-ary rooted tree of depth\footnote{We use here the convention that the depth is the 
  graph distance between the root and the leaves.}
$L$. Denote by $T_{\rm rel}(\bbT)$ and $ T_{\rm mix}(\bbT)$ the
relaxation time
of the process on $\bbT$ with no constraints for the spins
at the leaves (cf definitions  \ref{def:gap} and
\ref{def:mixing}). Then (cf Theorem \ref{th:j=k}) $T_{\rm rel}=\O(L^2)$ and $T_{\rm
  rel}=O(L^{2+\b})$ for some $0\le \b <\infty$. 
\item \emph{Quasi-critical case}. Assume $p=p_c-\epsilon$,
  $0<\epsilon\ll 1$,  and let $T_{\rm rel}$ be the relaxation
  time for the process on the infinite tree $\bbT^k$. Then (cf Theorem
\ref{th:2}) $T_{\rm rel}=\O(\epsilon^{-2})$ and $T_{\rm
  rel}=O\bigl(\epsilon^{-2-\alpha}\bigr)$ for some $\alpha\geq 0$.
 \item \emph{Mixing time.} We basically show (cf Theorem \ref{lemmamixing})
   that the mixing time on a finite $k$-ary rooted tree of depth $L$
   behaves like $L\times T_{\rm rel}$.
\end{itemize}
Our results, which are identical to those proved for
the Ising model on trees at the spin glass critical point \cite{Ding}, 
represent the first rigorous analysis of a kinetically constrained model
at criticality. As shown in \cite{MT}, our approach has a good chance
to apply also to other models with an ergodicity
phase transition, notably the North-East model on $\bbZ^2$ for which the
critical density
$p_c$ coincides with the oriented percolation threshold \cite{CMRT}. .     

% \begin{center}
%   \bf Sul caso seguente sono ancora perplesso !!
% \end{center}
% \begin{Theorem} Assume $2\le j\le k-1$. \\
% \emph{\bf Critical case  $p=p_c$}. \\
% % There exists $c,c'$ and $\b$ such that, for each $L$,  the
% % spectral gap of the process on the finite subtree with $L$ levels and maximal b.c. on
% % the leaves satisfies \[c L^{-\b}\le \gap \le c'/L.\]
% \emph{\bf Sub-critical case $p<p_c$}. 
% % Then there exists $a,a'$ and $\a$ such that, for each $L$,  the

% % spectral gap of the process on the finite subtree with $L$ levels and maximal b.c. on
% % the leaves satisfies \[a (p_c-p)^{?}\le \gap \le a' (p_c-p)^{\a}.\]

% \end{Theorem}

\subsection{Model, notation and background}

\subsubsection*{The graph.}
The model we consider is defined on the infinite rooted
$k$-ary tree $\bbT^k$ with root $r$ and vertex set $V$. For each $x\in
V$,  $\mathcal K_x$ will denote the set of its
$k$ children and $d_x$ its \emph{depth}, \ie  the graph distance
between $x$ and the root $r$. The finite $k$-ary subtree of $\bbT^k$
with $n$
levels is the set $\bbT^k_n=\{x\in \bbT^k:\ d_x\le n\}$.  
For $x\in \bbT^k_n$, $\bbT^k_{x,n}$ will
denote the $k$-ary sub-tree of $\bbT_n^k$ rooted at $x$ with depth
$n-d_x$, where $d_x$ is the depth of $x$. In other words the leaves of
$\bbT^k_{x,n}$ are a subset of the leaves of $\bbT_n^k$. We also set $\hat\bbT^k_{x,n}=\bbT^k_{x,n}\setminus \{x\}$
(See Figure 1 below). In the sequel, whenever no
confusion arises, we will drop the superscripts $k,n$ from $\bbT^k_n$
and $\bbT_{x,n}^k$.

\subsubsection*{The configuration spaces.} We choose as configuration space the set $\Omega=\{0,1\}^V$ whose
elements will usually be assigned Greek letters. We will often write $\h_x$ for the value
at $x$ of the element $\h\in \O$. We will also write
$\O_A$ for the set $\{0,1\}^A$, $A\subseteq V$.  With a slight abuse of notation, for
any $A\subseteq V$ and any $\eta,\o\in \O$, we let $\eta_A$  be
the restriction of $\eta$ to the set $A$ and $\eta_A\cdot\omega_{A^c}$  be the
configuration which equals $\eta$ on $A$ and $\omega$ on $V\setminus A$.

\subsubsection*{Probability measures.}
For any $A\subseteq V$ we denote
by~$\mu_A$ the product measure $\otimes_{x\in A}\,\mu_x$ where each
factor $\mu_x$ is the Bernoulli measure on $\{0,1\}$ with $\mu_x(1)=p$
and $\mu_x(0)=q$ with $q=1-p$. If $A=V$ we abbreviate $\mu_V$ to
$\mu$. 
Also, with a slight abuse of
notation, for any finite $A\subset V$, we will write $\mu(\eta_A)=\mu_A(\eta_A)$.

\subsubsection*{Conditional expectations and conditional variances.}
Given $A\sset V$ and a function $f \colon \Omega\to\bbR$ depending on finitely many variables,
in the sequel referred to as \emph{local function}, we define the
function $\eta_{A^c} \mapsto \mu_A(f)(\eta_{A^c})$ by the formula: 
\[
\mu_A(f)(\eta_{A^c}):= \sum_{\sigma\in\Omega_A}
\mu_A(\sigma)f(\sigma_A\cdot\eta_{A^c}).
\]
Clearly $\mu_A(f)$ coincides with the
\emph{conditional expectation} of $f$ given the configuration outside $A$.
Similarly we write $\Var_A(f)=\mu_A(f^2)-\mu_A(f)^2$ for
the {\it conditional variance} of~$f$ given $\eta_{A^c}$.  Note that $\Var_A(f)=0$
iff $f$ does not depend on the configuration inside~$A$. When $A=V$, respectively $A=\{x\}$ for some $x \in V$, we
abbreviate $\Var_V(f)$ to $\Var(f)$, respectively $\Var_{\{x\}}(f)$ to $\Var_x(f)$. 
%With a slight abuse of notation we shall write 
%$\mu_A(f)$ for $\mu_A(f)(\eta_{A^c})$, and similarly for the variance.
\begin{definition}[OFA-kf model]\label{defFA}
The {\rm OFA-kf} (Oriented
  Fredrickson-Andersen k-facilitated)  model at density $p$ is a continuous time Glauber type Markov
processe on $\O$, reversible w.r.t.\ $\mu$, 
with Markov semigroup $P_t=\nep{t\cL}$ whose infinitesimal generator $\cL$ acts
on local functions $f:\O\mapsto \mathbb R$ as follows: 
\begin{align}
\cL f(\omega)&=\sum_{x\in \bbT^k}c_{x}(\o)\left[\mu_x(f)(\o)-f(\o)\right] .
\label{FAjf}
\end{align}
The function $c_x$, in the sequel referred to as the
\emph{constraint at $x$}, is defined by
\begin{align}
\label{cx}
   c_{x}(\o)&=
    \begin{cases}
   1 & \text{if  $\o_y=0\ \forall y\in \mathcal K_x$}\\
  0 & \text{otherwise}.
    \end{cases}
  \end{align}
\end{definition}
It is easy to check by standard methods (see \textit{e.g.}\ \cite{Liggett}) that
the process is well defined and that its generator can be extended to
non-positive self-adjoint operators on $L^2(\bbT^k,\mu)$.

The OFA-kf process can of course be defined also on finite rooted trees. In this case and in order to ensure irreducibility
of the Markov chain the constraints $c_x$ must be suitably modified.

\begin{definition}[Finite volume dynamics]\label{finite}
Let $\bbT$ be a finite subtree of $\bbT^k$ and let, for any $\eta\in \O_\bbT$, $\eta^{0}\in \O$ denote the
extension of $\eta$ in $\O$  given 
by
\begin{equation*}
\eta^0_x=   \begin{cases}
 \eta_x &\text{if $x\in \bbT$} \\
0 &\text{if $x\in \bbT^k\setminus \bbT$}.   
  \end{cases}
\end{equation*}
For any $x\in \bbT$ define the \emph{finite constraints} $c_{\bbT,x}$ by
\begin{equation} \label{finitecx}
c_{\bbT,x}(\eta)= c_x(\eta^0).
\end{equation}
We will then consider the irreducible, continuous time
Markov chains on $\O_\bbT$ with generator 
\begin{align}
\label{fin-gen2}
\cL_\bbT f=\sum_{x\in \bbT}c_{\bbT,x}[\mu_x(f)-f] \qquad \eta
\in \O_\bbT. \end{align} 
\end{definition}
Note that irreducibility of the above defined finite volume dynamics
is guaranteed by the fact that starting from the empty leaves one can empty any configuration via allowed spin flips.
It is natural to define (see \cite{CMRT}) the critical density for
the model by:
\begin{align} %\label{eq:1bis}
p_c =\sup\{p\in[0,1]:\text{0 is simple eigenvalue of } \cL\}
\end{align}
The regime $p<p_c$ is called the {\sl ergodic region} and we say that an {\sl ergodicity breaking transition} occurs at the
critical density. In \cite{MT} it has been established that $p_c$
coincides with the percolation threshold $1/k$
and that for all $p<p_c$ the value $0$ is a simple eigenvalue of the
generator $\cL$. Actually much more is known but first we need to introduce
some relevant time scales.
 
\begin{definition}[The relaxation time] \label{def:gap}
Let $\cD(f) :=\mu(f,-\cL f)$ be the Dirichlet form corresponding to the
generator $\cL$. We define the spectral
 gap of the process as
\begin{equation}
  \label{eq:gap}
\gap(\cL):=\inf_{\substack{f\in {\rm Dom}(\cL)\\ f\neq \text{\rm const}}}\frac{\cD(f)}{\Var(f)}
\end{equation}
We also define the \emph{relaxation time} by $T_{\rm rel}:=
\gap(\cL)^{-1}$. Similarly, if $\bbT$ is a finite rooted tree,
we define  $T_{\rm rel}(\bbT):= \gap(\cL_\bbT)^{-1}$.
\end{definition}

\begin{definition}[Mixing times]
\label{def:mixing}
Let $\bbT$ be a finite rooted sub-tree of $\bbT^k$. For any $\eta\in \O_\bbT$  we denote by $\nu_t^\eta$ 
the law at time $t$ of the Markov chain with
generator $\cL_{\bbT}$ and by $h_t^\eta$ its relative density w.r.t.\ $\mu_\bbT$. Following \cite{Gine}, we define the family of mixing times $\{T_a(\bbT)\}_{a\ge 1}$ by
\[
T_a(\bbT):= \inf\left\{t\ge 0:\ \max_\eta\mu_\bbT\left(|h_t^\eta -1|^a\right)^{1/a}\le 1/4\right\}.
\]
Notice that $T_1(\bbT)$ coincides with the usual \emph{mixing time} $T_{\rm mix}(\bbT)$ of the chain  (see
\textit{e.g.} \cite{Peresbook}) and that, for any $a\ge 1$, $T_1\le T_a$. 

\end{definition}
With the above notation it was proved in \cite{MT} that
\begin{enumerate}[(i)]
\item for all
$p<p_c$, $T_{\rm rel}<+\infty$ and that the mixing time on a finite
regular $k$-ary sub-tree of depth $L$ grows linearly in $L$;
\item if $p>p_c$, then both the relaxation time and the mixing time on a finite
regular $k$-ary sub-tree of depth $L$ grow exponentially fast in $L$.  
\end{enumerate}
\subsection{Main Results}\ 
\\
Our first contribution concerns the critical case $p=p_c$.
\begin{MainTheorem}\label{th:j=k} 
Fix $k\geq 2$ and assume $p=p_c$. 
Then there exist constants $c>0$ and $\b\ge 0$, with $\beta$ independent of $k$, such that for
each $L$
\begin{align*}
  c^{-1} L^2\le &T_{\rm rel}\bigl(\bbT_L^k\bigr)\le cL^{2+\beta}. \\
%c^{-1}L \le &T_{\rm mix}(\bbT_L)\le c L^{1+\beta}.
\end{align*}
\end{MainTheorem}
\begin{remark}
The above result implies, in particular, that the relaxation time for
the critical process on
the infinite tree $\bbT^k$ is infinite. However the process is still
ergodic in the sense that $0$ is a simple eigenvalue of the
 generator $\cL$. This can be proven following the same lines of
 \cite{CMRT}*{Proposition 2.5} by using the key ingredient that, at
 $p=p_c$, there is no infinite percolation of occupied vertices a.s..
% the probability that a given site belongs to an infinite
%  cluster of all $1$'s is zero ($p_n\to 0$ with the notation of Lemma \ref{lemma:A} below).
\end{remark}
Our second main result deals with the quasi-critical regime,
$p=p_c-\epsilon$ with $0<\epsilon\ll 1$, on the infinite tree $\bbT^k$.
\begin{MainTheorem} \label{th:2}
Fix $k\geq 2$ and assume $p<p_c$.
 Then there exist constants $a>0$ and $\a\ge 0$, with $\alpha$ independent of $k$, such that
\begin{align*}
  a^{-1} (p_c-p)^{-2}\leq&T_{\rm rel}\leq a(p_c-p)^{-(2+\a)} \\
%a^{-1} (p_c-p)^{-1}\leq &T_{\rm mix}(\bbT^k)\leq a (p_c-p)^{-\a-1} .
\end{align*}
\end{MainTheorem}
The last result derives some consequences of the above theorems for the
mixing time on a finite sub-tree.

\begin{MainTheorem}\label{lemmamixing}
%[\cite{MT}]
There exists $c>0$ such that, for all $L$,
 \begin{equation}
   \label{eq:4}
\frac 1c L\, T_{\rm rel}\bigl(\bbT^k_{\lfloor L/2\rfloor}\bigr)\le T_1(\bbT_L^k)\le T_2(\bbT_L^k) \le
cL\,T_{\rm rel}(\bbT_L^k).
 \end{equation}
\noindent
In particular:
\begin{enumerate}[(i)]
\item if $p=p_c\,$, then 
\[
c^{-1}L^3 \le T_1(\bbT_L^k)\le cL^{3+\beta}.
\] 
\item
If $p<p_c\,$,
\[
\frac 1c (p_c-p)^{-2} L \le T_1(\bbT_L^k) \le c L (p_c-p)^{-(2+\alpha)}
\]  
\end{enumerate}
for some constants $\a, \b\ge 0$ independent of $L$.
\end{MainTheorem}

\subsection{Additional notation and technical preliminaries} 

We first introduce the natural bootstrap map for the model.
\begin{definition}
The bootstrap map $B:\{0,1\}^{\bbT^k}\to
\{0,1\}^{\bbT^k}$ associated to the OFA-kf model is defined by
\begin{equation}
  \label{eq:bootstrap map}
  B(\h)_x=
  \begin{cases}
0 & \text{if either  
$\h_x=0$ or $c_x(\h)=1$}\\
1 & \text{otherwise}
  \end{cases}
\end{equation}
with $c_x$ defined in \eqref{cx}. 
\end{definition}
\begin{remark}
Notice that: (i) if after $n$-iterations of the bootstrap map $c_x(B^n(\eta))=1$ then, even if $\eta_x=1$, the
percolation cluster of $1$'s attached to $x$ is contained in the
first $n$-levels below $x$ and (ii) the bootstrap critical point (see
e.g. \cite{Boot-Peres}) coincides with the percolation threshold $p_c=1/k$.
\end{remark}
Secondly we formulate two technical results which will be useful in the sequel. Let
$E^{(n)}_x=\{\eta:\ B^n(\eta)_x=1\}$ and define
$p_n:=\mu(E^{(n)}_r)$. Notice that $p_n$ is increasing in $p$ and that
$p_n\le p$ for all $n$.
\begin{Lemma}
\label{lemma:A}\ 
\begin{enumerate}[(i)]
\item If $p\leq p_c$ then $p_n\leq \frac{2}{(k-1)n}$ for all $n\geq
  1$.
\item Assume $p=p_c-\epsilon$ with $\epsilon\in[0,1/k]$.  Then
  $p_n\leq p(1-\epsilon k)^n$ for all $n\geq
  1$. 
\end{enumerate}
% (i) \\
% (iii) \\
% (iv) Assume $p=p_c-\epsilon$ with $\epsilon\in[0,1/k)$. Then there exists $c>0$ such that
% $$(1-\epsilon k)^n\prod_{i=1}^n (1-p_i)^{k-1}\leq c \,\frac{p_n^2}{(1-\epsilon k)^{n+2}}.$$
% %Let $\epsilon>0$ and $p= p_c-\epsilon$ then for any $\delta>0$ there exists $n_{\delta},\epsilon_{\delta},c_{\delta}$ s.t.
% %for $\epsilon\in[0,\epsilon_{\delta}]$, $n\in [n_\delta, c_\delta/(\epsilon k)]$ it holds
% %$p_n\geq 2(1-\delta)/[(k-1)n]$.
\end{Lemma}
\begin{proof}\
\smallskip
 
(i) Using the monotonicity in $p$ of the $p_n$'s it is enough to prove
the statement for $p=p_c$. We start from  
\begin{equation}
  \label{eq:3}
\mu\left(E^{(n+1)}_r\right)=p\mu\left(\cup_{x\in \cK_r}E^{(n)}_x\right),
\end{equation}
or, equivalently, 
\[
p_{n+1}=p(1-(1-p_n)^k).
\] Using inclusion-exclusion inequalities \eqref{eq:3} implies (recall
that $p=1/k$) 
\begin{align}
  \label{eq:2}
p_{n+1}&\le \frac 1k \left[k p_n - {k\choose 2}p_n^2 + {k\choose
    3}p_n^3\right] \nonumber\\&= p_n - \frac{(k-1)}{2}p_n^2
+\frac{(k-1)(k-2)}{6} p_n^3.
\end{align}
One readily checks that the r.h.s. of \eqref{eq:2} is increasing in
$p_n\in [0,1/k]$. Thus, if we assume inductively that $p_n \le
\frac{2}{(k-1)n},\ n\ge 2$, we obtain
\begin{align*}
 p_{n+1}\le
\frac{2}{(k-1)}\left[\frac 1n-\frac {1}{n^2} +
   \frac{2(k-2)}{3(k-1)n^3}\right]\le \frac{2}{(k-1)(n+1)} \quad n\ge 2.
\end{align*} 
The base case $p_2$ follows from the trivial observation that $p_2\le
p_1\le \frac 1k < \frac{1}{k-1}$.
\smallskip

(ii) Boole inequality applied to \eqref{eq:3} gives
\begin{align*}
  \label{eq:3}
p_{n+1}\le pkp_n= (1-\epsilon k)p_n\le \ldots\le (1-\epsilon k)^np. 
\end{align*}

\end{proof}
The second technical ingredient is the following monotonicity result for the spectral gap (see \cite {CMRT}*{Lemma 2.11} for a proof).
\begin{Lemma}
\label{gapmon}
Let $\bbT_1 \subset \bbT_2$ be two sub-trees of $\bbT^k$. Then,
$$
\gap({\mathcal{L}}_{\bbT_1}) \geq \gap({\mathcal{L}}_{\bbT_2}).
$$
\end{Lemma}

\section{The critical case: proof of Theorem \ref{th:j=k}}
\label{section:critical}
\subsection{Upper bound of the relaxation time} 
Let $\bbT\equiv \bbT^k_L, \bbT_x\equiv T^k_{x,L}$ and $\hat
\bbT_x\equiv \hat T^k_{x,L}$. We divide the proof of the upper bound of
$T_{\rm rel}(\bbT)$ in three steps.
\smallskip

\noindent 
\subsubsection{First step.}[Comparison with a long-range auxiliary dynamics].
Motivated by \cite{MT} we introduce auxiliary  {\sl long range}
constraints as follows.
\begin{definition}
\label{meso}
For any integer $\ell\geq 1$ we set 
\begin{equation*}
  c_x^{(\ell)}(\h)=
  \begin{cases}
    1 &\text{ if
      $ c_x(B^{\ell-1} (\eta))=1$}\cr
%      &\text{if the depth of $\bbT_x$ is smaller than $\ell$ or if
    %  $ c_x(B^\ell (\eta))=1$}\cr
0 &\text{otherwise.}
  \end{cases}
\end{equation*}
\end{definition}
\begin{remark}
One can use the functions $c_x^{(\ell)} $ to define an auxiliary long range dynamics with generator given by  \eqref{FAjf} with  $c_x$ replaced 
by $c_x^{(\ell)}$. For this new constrained dynamics 
a vertex $x$
is free to flip
iff, by a sequence of at most $\ell$ flips satisfying the original constraints
\eqref{cx} 
all the children of $x$ can be made
vacant. 
\end{remark}
Fix now $\d\in (0,1/9)$ and choose $\ell = (1-\d)L$ (neglecting integer
part). Let also $c_{\bbT,x}^{(\ell)}(\eta):=c_{x}^{(\ell)}(\eta^0)$ where
$\eta^0$ is given in Definition
\ref{finite} respectively. Notice that
$c_{\bbT,x}^{(\ell)}(\eta)\equiv 1$ iff $d_x> L-\ell$. 
%To simplify notation we will drop the subscript $\bbT$ in $c_{\bbT,x}^{(\ell)}$. 
We will establish the inequality
\begin{equation}
  \label{PI}
\Var_\bbT(f)\le \l \sum_{x\in \bbT} \mu_\bbT\left(\Var_x(\mu_{\hat
    \bbT_x}(c_{\bbT,x}^{(\ell)}f)\right)\qquad \forall f
\end{equation}
with $\l= 2(\frac{1-\d}{1-9\d})$.
\begin{remark}
Inequality \eqref{PI} will be proven following the strategy of \cite{MT}. Notice however that here we don't perform another Cauchy-Schwartz
inequality to pull out the constraint $c_{\bbT,x}^{(\ell)}$ and get the
 Dirichlet form with long range constraints.
\end{remark}
We  start from
\begin{equation} \label{eq:1}
\Var_\bbT(f)\le \sum_{x\in \bbT} \mu_\bbT\left(\Var_x(\mu_{\hat \bbT_x}(f)\right).
\end{equation}
The above inequality follows easily from a repeated use of the formula
for conditional variance and we refer to section 4.1 in \cite{MT} for
a short proof.
We now examine a generic term $\mu\(\Var_x\(\mu_{\hat \bbT_x}(f)\)\)$
in the r.h.s.\ of \eqref{eq:1}. 
%If $d_x> \ell\delta/(1-\delta)$ then
%$c_{\bbT,x}^{(\ell)}\equiv 1$ by definition and we can safely write
 %\[\sum_{x:\  d_x> \ell\delta/(1-\delta)} \mu_\bbT\left(\Var_x(\mu_{\hat \bbT_x}(f)\right)=
%\sum_{x:\  d_x>\ell\delta/(1-\delta)} \mu_\bbT\left(\Var_x(\mu_{\hat \bbT_x}(c_{\bbT,x}^{(\ell)}f)\right) .
%\]
%If instead $d_x\leq\ell\delta/(1-\delta) $ 
We write 
\begin{equation*}
 \mu_{\hat \bbT_x}(f)=\mu_{\hat \bbT_x}\(c_{\bbT,x}^{(\ell)}f\) + \mu_{\hat \bbT_x}([1-c_{\bbT,x}^{(\ell)}]f) 
\end{equation*}
so that
\begin{equation}
  \Var_x\( \mu_{\hat \bbT_x}(f) \) \le 2\Var_x\(\mu_{\hat
    \bbT_x}\(c_{\bbT,x}^{(\ell)}f\) \) + 2 \Var_x\(\mu_{\hat \bbT_x}\((1-c_{\bbT,x}^{(\ell)})f\)\) .
\label{D0}\end{equation}
We now consider the second term $\Var_x\(\mu_{\hat
  \bbT_x}\((1-c_{\bbT,x}^{(\ell)})f\)\)$. Without loss of generality we can
assume $\mu_{\hat \bbT_x}(f)=0$. 
Recall that the constraint $c_{\bbT,x}^{(\ell)}$ depends only on the spin configuration in the first $\ell$ levels below $x$, in the sequel denoted by $\D_x$ (see Figure \ref{fig:albero1}). 

\begin{figure}[ht] 
\psfrag{L}{$L$}
\psfrag{L2}{$L-d_y$}
\psfrag{x}{$x$}
\psfrag{y}{$y$}
\psfrag{D}{$\Delta_x$}
\psfrag{Ty}{$\hat \bbT_y$}
\psfrag{l}{$\ell$}
\psfrag{r}{$r$}
\includegraphics[width=.70\columnwidth]{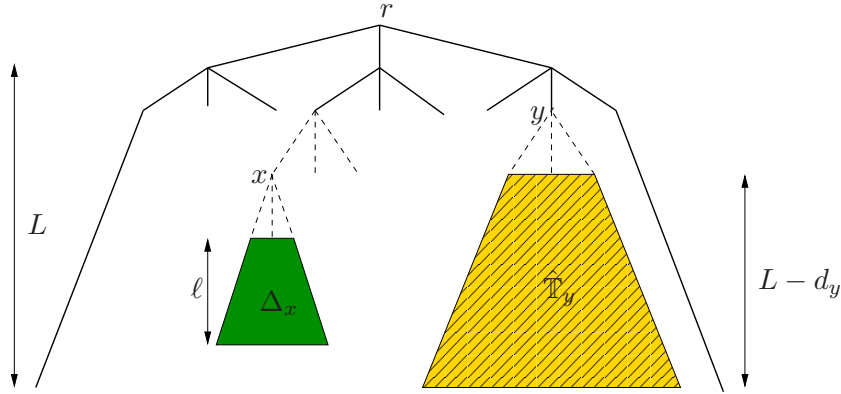}
\caption{For $k=3$, the tree $\bbT$ rooted at $r$, of depth $L$ (\textit{i.e.}\ with $L$ levels below $r$), the set $\Delta_x$  and the sub-set $\hat \bbT_y$.}
\label{fig:albero1}
\end{figure}

Thus
\[
\mu_{\hat \bbT_x}\((1-c_{\bbT,x}^{(\ell)})f\)= \mu_{\hat \bbT_x}\((1-c_{\bbT,x}^{(\ell)})\mu_{\hat\bbT_x\setminus \D_x}(f)\)
\]
and 
\begin{align}
   \Var_x\(\mu_{\hat \bbT_x}\((1-c_{\bbT,x}^{(\ell)})f\)\) &\le \mu_{\bbT_x}\(\mu_{\hat \bbT_x}\((1-c_{\bbT,x}^{(\ell)})\mu_{\hat\bbT_x\setminus \D_x}(f)\)^2\)
  \nonumber \\
&\leq \mu_{\bbT_x}(1-c_{\bbT,x}^{(\ell)})
\mu_{\bbT_x}\(\mu_{\hat\bbT_x\setminus \D_x}(f)^2\) \nonumber\\ 
&=
\mu_{\bbT_x}(1-c_{\bbT,x}^{(\ell)}) \Var_{\bbT_x}\( \mu_{\hat\bbT_x\setminus
  \D_x}(f)\)\nonumber\\
&\le \mu_{\bbT_x}(1-c_{\bbT,x}^{(\ell)})\sum_{y\in \D_x\cup x}\mu_{\bbT_x}\left(\Var_y(\mu_{\hat \bbT_y}(f)\right)
 \label{D1}
\end{align}
where we used Cauchy-Schwartz inequality, the fact that $c_{\bbT,x}^{(\ell)}$ does
not depend on $\h_x$ and \eqref{eq:1} in the last inequality.  From the definition of $c_{\bbT,x}^{(\ell)}$ on the finite tree $\bbT$ it holds
\begin{equation}
  \label{casicx}
 \mu_{\bbT_x}(1- c_{\bbT,x}^{(\ell)})=
  \begin{cases}
0 & \text{if $d_x>\delta L$}\\
p_{\ell}/p& \text{otherwise}
  \end{cases}
\end{equation}

%$\mu_{\bbT_x}(1-c_r^{(\ell)})=p_\ell/p= k p_\ell \leq\frac{8k}{(k-1)(\ell+1)}$.

In conclusion, using \eqref{D0}, \eqref{D1}and \eqref{casicx},
\begin{gather}
\label{altra}
\ \sum_{x\in\bbT} \!\mu_\bbT\left[\Var_x(\mu_{\hat \bbT_x}(f))\right] \\ 
 \le 2\sum_{x\in\bbT}\!\mu_\bbT\left[\Var_x(\mu_{\hat \bbT_x}(c_{\bbT,x}^{(\ell)}f)\right]
   +  2\frac{p_{\ell}}{p}\sum_{\genfrac{}{}{0pt}{}{x:}{d_x \leq\delta
       L}}\sum_{y\in \D_x\cup x}\mu_{\bbT}[\Var_y(\mu_{\hat
     \bbT_y}(f)]\nonumber \\
\le 2\sum_{x\in\bbT}\!\mu_\bbT\left[\Var_x(\mu_{\hat \bbT_x}(c_{\bbT,x}^{(\ell)}f)\right]
   +  2\frac{p_{\ell}}{p}\left[\max_z N_z \right]\sum_{y}\mu_{\bbT}[\Var_y(\mu_{\hat
     \bbT_y}(f)]
%32\frac{\delta}{1-\delta}\sum_{y\in\bbT} \mu_\bbT\left[\Var_y(\mu_{\hat \bbT_y}(f)\right] 
  \end{gather}
where 
\[
N_z:=\#\{x:\ \D_x \ni z,\ d_x\le \d L\}\le \min(\d L,\ell+1).
\] 
Part (i) of Lemma \ref{lemma:A} implies that $p_{\ell}\le \frac{2}{(k-1)\ell}=\frac{2}{(k-1)(1-\delta)L}$ so that
 % the r.h.s. of
%\eqref{altra} can be bounded by
\begin{gather}
\label{sabato} \sum_{x\in\bbT} \mu_\bbT\left(\Var_x(\mu_{\hat
    \bbT_x}(f)\right)\nonumber\\
\leq
2\sum_{x\in\bbT}\!\mu_\bbT\left[\Var_x(\mu_{\hat
    \bbT_x}(c_{\bbT,x}^{(\ell)}f)\right]+ \frac{4\delta
}{p(1-\delta)(k-1)}\sum_{x\in\bbT}\mu_{\bbT}[\Var_x(\mu_{\hat
  \bbT_x}(f)]
\end{gather}
Since $p=1/k$ and $k/(k-1)\leq 2$, inequality \eqref{PI} holds with $\l=
2(1-\d)/(1-9\d)$ provided  $8\d/(1-\d)<1$.
\smallskip

\noindent 
\subsubsection{Second step}[Analysis of the auxiliary dynamics].
 Let $h_i= \a^i$, 
$\a>1$ to be fixed later on, and
%  and let
% $n_x=\max\{i:\ h_i\le \min(\ell,L-d_x)\},\ x\in \bbT$.
let 
\begin{equation}
\label{Ti}
T_{i}:= T_{\rm rel}(\bbT^k_{h_{i}\wedge \ell}).
\end{equation}
We shall now prove that 
\begin{equation}\label{eq2.8}
\sum_{x\in \bbT} \mu_\bbT  \left(\Var_x( \mu_{\hat \bbT_x}(c_{\bbT,x}^{(\ell)}f))\right) 
\leq 
\left[2+ \frac{4\a  }{p(k-1)} \left(\sum^{n-1}_{i=1} \sqrt{T_{i}}\right)^2 \right]\cD_\bbT(f),
\end{equation}
with $n$ such that $h_{n-1}< \ell \le h_n$. 

The starting point is \eqref{PI}. For any $x\in \bbT$ we
introduce a scale decomposition of the constraint $c_{\bbT,x}^{(\ell)}$ as
follows $c_{\bbT,x}^{(\ell)}=\sum_{i=0}^{n-1}\chi_i +c_{\bbT,x}$,
where $\chi_i:= c_{\bbT,x}^{(h_{i+1}\wedge
  \ell)}-c_{\bbT,x}^{(h_{i}\wedge \ell)}$. 
Thus 
\begin{gather*}
\sum_{x\in \bbT} \mu_\bbT  \left(\Var_x( \mu_{\hat \bbT_x}(c_{\bbT,x}^{(\ell)}f))\right) 
\\\leq 2 \sum_{x\in \bbT} \mu_\bbT  \left(\Var_x( \mu_{\hat
    \bbT_x}(c_{\bbT,x} f))\right) +
2 \sum_{x\in \bbT} \mu_\bbT  \left(\Var_x( \mu_{\hat
    \bbT_x}(\sum_{i=0}^{n-1}\chi_i f))\right)\\
\leq 2\cD_{\bbT}(f) + 2 \sum_{x\in \bbT} \mu_\bbT  \left(\Var_x( \mu_{\hat
    \bbT_x}(\sum_{i=0}^{n-1}\chi_i f))\right),
  \end{gather*}
where in the last inequality we used convexity to conclude that
\[
\mu_\bbT  \left(\Var_x( \mu_{\hat\bbT_x}(c_{\bbT,x} f))\right)\le
\mu_\bbT\left(c_{\bbT,x}\Var_x(f)\right).
\]
We now examine the key term $\sum_{x\in \bbT} \mu_\bbT  \left(\Var_x( \mu_{\hat
    \bbT_x}(\sum_{i=0}^{n-1}\chi_i f))\right)$.

Observe first that $\chi_i=0$ if
$h_i\ge \ell$ and that  
$\chi_i=1$ implies the number of iterations of the bootstrap map necessary
to make the node $x$ flippable is at least 
$h_i$ but no more than $h_{i+1}\wedge \ell$. In particular, if $\chi_i(\eta)=1$,
there exists a ``line'' of  zeros of $\eta$ within $h_{i+1}\wedge \ell$ levels
below $x$. For such an $\eta$ we denote by $\G(\h)$ the ``lowest''
such line constructed as follows. \\
Consider the nodes in $\bbT_x$ at
distance $h_{i+1}\wedge \ell$ from $x$. Let us order them from left to right as
$z_1,z_2,\dots$; start from $z_1$ and find the first empty site on the
branch leading to $x$. Call this vertex $y_1$ and forget about all the
$z_i$'s having $y_1$ as ancestor. Say that the remaining nodes are $z_{k_1},z_{k_1+1},\dots$;
repeat the construction for $z_{k_1}$ to get a new empty node
$y_2$ and so forth. At the end of this procedure some of the $y_i$
may have some other $y_k$ as ancestor. In this case we remove the former
from our collection and we relabel accordingly. The line $\G(\h)$ is then the final collection
$(y_1,y_2,\dots)$. 

We denote by $\cG_i$ the space of all possible
realisations of $\G$. Moreover, given $\g\in \cG_i$, we denote by
$\hat\bbT_x^{\g,+}$ all the nodes in $\hat\bbT_x$ which have no ancestor
  in $\g$, \ie the part of the tree ``above'' $\g$. 
%\begin{itemize}
%\item $\hat\bbT_x^{\g,+}$ all the nodes in $\hat\bbT_x$ which have no ancestor
%  in $\g$, \ie the part of the tree ``above'' $\g$;
%\item $\hat\bbT_x^{\g,-}$ all the nodes in $\hat\bbT_x$ which have an ancestor
 % in $\g$, \ie the part of the tree ``below'' $\g$;
%\item $I_\g$ all the nodes in $\bbT^{\g,-}_x$ such that their distance from $x$ is at most
%$h_{i+1}$.
%\end{itemize}
%The key property of the above construction is that, for any $\g\in
%\cG_i$, conditionally on
%the event $\{\h:\ \G(\h)=\g\}$, the variables $\{\h_y\}_{y\in
 %\bbT_x\setminus (I_\g\cup \g)}$ are still i.i.d.\ $p$-Bernoulli variables.  
Note that the above construction of $\Gamma$ is made without looking at the configuration above $\Gamma$.
\begin{figure}[h] 
\psfrag{L}{$L-dx$}
\psfrag{hi}{$h_i$}
\psfrag{x}{$x$}
\psfrag{hi+1}{$h_{i+1}$}
\psfrag{Ig}{$I_\gamma$}
\psfrag{T+}{$\bbT_x^{\g,+}$}
\psfrag{T-}{$\bbT_x^{\g,-}$}
\includegraphics[width=1\columnwidth]{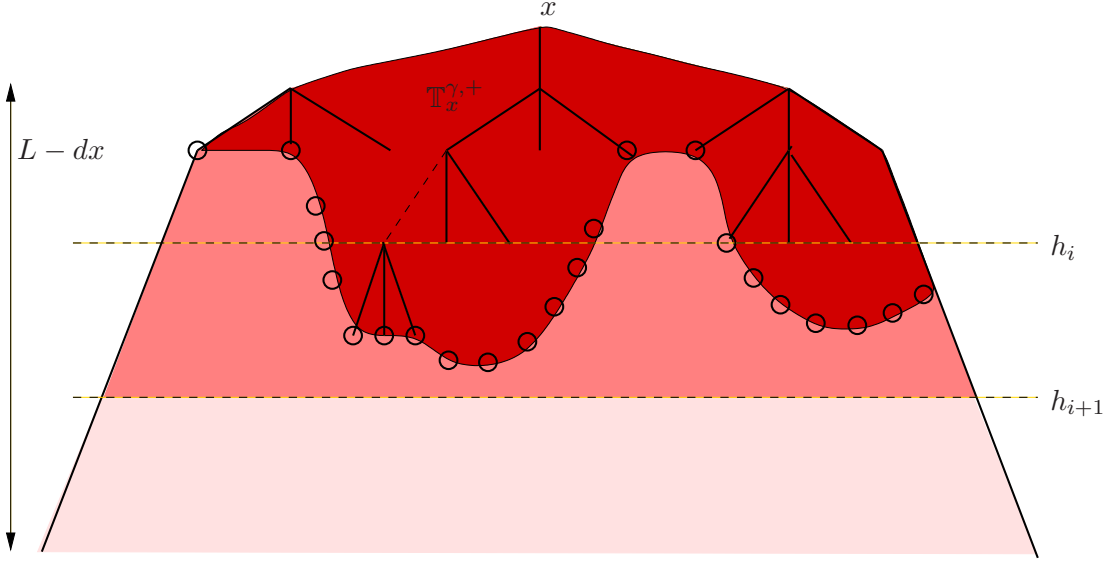}
\caption{For $k=3$, the sub-tree $\bbT_x$ rooted at $x$ and a configuration $\eta$ such that $\chi_i(\eta)=1$. The line of empty sites corresponds to a set $\gamma \in \cG_i$.
}
\label{fig:albero2}
\end{figure}
This observation together with the definition of the variance and Cauchy-Schwarz inequality gives
\begin{align}
\Var_x\left(\mu_{\hat
    \bbT_x}(\sum_{i=0}^{n-1}\chi_i f)\right) 
    &=  p(1-p)\left[\,\sum_{i=0}^{n-1} \mu_{\hat
    \bbT_x}(\chi_i \nabla_x f)\,\right]^2 \nonumber\\
&=p(1-p)\left [ \sum_{i=0}^{n-1} \sum_{\g\in
   \cG_i}\mu_{\hat\bbT_x\setminus \hat T_x^{\gamma,+}}\left(\id_{\G=\g}\;\mu_{ \hat T_x^{\gamma,+}}\left(\chi_i
\nabla_x f \right)\right)\right]^2 \\
&\le p(1-p)\left [ \sum_{i=0}^{n-1} \sum_{\g\in
   \cG_i}\mu_{\hat\bbT_x\setminus \hat T_x^{\gamma,+}} \left(\id_{\G=\g}\sqrt{\mu_{\hat T_x^{\gamma,+}}(
\chi_i) \mu_{\hat T_x^{\gamma,+}}(|\nabla_x f|^2) }\right)\right]^2.
\nonumber   \label{SQ}
\end{align}
where $\nabla_xf(\eta)=f(\eta^x)-f(\eta)$ with $\eta^x_y= \eta_y$ if
$y\neq x$ and $\eta^x_x=1-\eta_x$. 
Consider now the last factor inside the square root and multiply it by $p(1-p)$. It holds
 \begin{gather*}
 p(1-p) \mu_{T_x^{\gamma,+}}(|\nabla_x f|^2)=\mu_{T_x^{\gamma,+}}(\Var_x(f))\leq \Var_{\bbT_x^{\g,+}}(f)
\le T_{\rm rel}(\bbT_x^{\g,+})\cD_{\bbT_x^{\g,+}}(f) 
     \end{gather*}
     where we used the convexity of the variance and the Poincar\'e inequality.
Lemma \ref{gapmon} now gives $
T_{\rm rel}(\bbT_x^{\g,+})\le T_{i+1}.
$
 In conclusion 
\[
p(1-p)
\mu_{T_x^{\gamma,+}}(|\nabla_x f|^2)
\le
T_{i+1}\cD_{\bbT_x^{\g,+}}(f).
\]
To bound  the first factor inside the square root of \eqref{SQ} we
note that $\id_{\G=\g}c_{\bbT,x}^{(h_i)}=\id_{\G=\g}c_{\bbT_x^{\g,+},x}^{(h_i)}$.
Indeed the finite volume constraints $c_{T_x^{\gamma,+},y}$ are defined with zeros on the set $\gamma$ of the leaves of $T_x^{\gamma,+}$  (see \eqref{finitecx}) and in turn $\id_{\Gamma(\eta)=\gamma}$  guarantees the presence of such zeros for the configuration  $\eta$.
Thus, using the monotonicity on the volume of the probability that the root $x$ is connected to the level $h_i$,
%c'era : 
%$$
%\mu_{T_x^{\gamma,+}}(\chi_i)\leq\mu_{T_x^{\gamma,+}}(1-c_x^{(h_i)})=\mu(1-c_x^{(h_i)})\leq p_{h_i}/p
%$$
\begin{gather*}
\id_{\G=\g}\mu_{T_x^{\gamma,+}}(\chi_i)\leq\id_{\G=\g}\mu_{T_x^{\gamma,+}}(1-c_x^{(h_i)})\\=\id_{\G=\g}\mu_{T_x^{\gamma,+}}(1-c_{T_x^{\gamma,+},x}^{(h_i)})\leq \mu(1-c_x^{(h_i)})=p_{h_i}/p.
\end{gather*}
In conclusion, the r.h.s.\ of \eqref{SQ} is bounded from above by
\begin{align}
\label{basic}
&\frac{1}{p}\left( \sum_{i=0}^{n-1} \sqrt{ T_{i+1}p_{h_i}} \; \mu_{\hat \bbT_x}\Bigl(\sum_{\gamma\in\mathcal G_i}\id_{\Gamma=\gamma}\sqrt{\cD_{\bbT_x^{\gamma,+}}(f)}\Bigr) \right)^2 \nonumber \\
\leq
&\frac{1}{p}\left( \sum_{i=0}^{n-1} \sqrt{ T_{i+1}p_{h_i}} \; \sqrt{\mu_{\hat \bbT_x}\Bigl(\sum_{\gamma\in\mathcal G_i}\id_{\Gamma=\gamma}\cD_{\bbT_x^{\gamma,+}}(f)}\Bigr)
 \right)^2 \nonumber \\
\leq 
&\frac{1}{p} \left(\sum_{i=0}^{n-1} \sqrt{T_{i+1}}\right) \left(\sum_{i=0}^{n-1} \sqrt{T_{i+1}}\;p_{h_i}\;\mu_{\hat
    \bbT_x}\Bigl(\sum_{\gamma\in\mathcal G_i}\id_{\Gamma=\gamma}\cD_{\bbT_x^{\gamma,+}}(f)\Bigr)\right)\nonumber\\
\leq 
&\frac{1}{p} \left(\sum_{i=0}^{n-1} \sqrt{T_{i+1}}\right) \left(\sum_{i=0}^{n-1}
  \sqrt{T_{i+1}}\, p_{h_i} \sum_{y\in \hat \bbT_x\atop d_y\le d_x+h_{i+1}}\mu_{\hat \bbT_x}(c_y \Var_y(f))\right)
\end{align}
where we used the Cauchy-Schwarz inequality in the first and second
inequality together with 
 \[
 \mu_{\hat
    \bbT_x}\Bigl(\sum_{\gamma\in\mathcal G_i}\id_{\Gamma=\gamma}\cD_{\bbT_x^{\gamma,+}}(f)\Bigr)
\leq \sum_{y\in \hat \bbT_x\atop d_y\le d_x+h_{i+1}}\mu_{\hat\bbT_x}(c_{\bbT,y} \Var_y(f)) 
\]
because $\id_{\Gamma(\eta)=\gamma}c_{T_x^{\gamma,+},y}(\eta)= \id_{\Gamma(\eta)=\gamma} c_{\bbT,y}(\eta)$. 
If we now average over $\mu_{\bbT}$ and sum over $x\in\bbT$ the above result we get that
\begin{align*}
 &\sum_{x\in \bbT} \mu_\bbT  \left(\Var_x( \mu_{\hat \bbT_x}(\sum_{i=0}^{n-1}\chi_if))\right) \\
& \le \frac{1}{p}
 \left(\sum_{i=0}^{n-1} \sqrt{T_{i+1}}\right) \left(\sum_{i=0}^{n-1}
  \sqrt{T_{i+1}}\, p_{h_i}\sum_{x\in \bbT} \sum_{y\in \hat \bbT_x\atop d_y\le d_x+h_{i+1}}\mu_\bbT(c_{\bbT,y} \Var_y(f))\right)\\
&\le \frac{1}{p}
 \left(\sum_{i=0}^{n-1} \sqrt{T_{i+1}}\right) \left(\sum_{i=0}^{n-1}
  \sqrt{T_{i+1}}\, p_{h_i}h_{i+1}\right)\cD_\bbT(f)\\
&\le \frac{2\a}{p(k-1)} \left(\sum_{i=0}^{n-1} \sqrt{T_{i+1}}\right)^2 \cD_\bbT(f)
\end{align*}
and \eqref{eq2.8} follows. Above we used the exponential growth of the
scales $\{h_i\}_i$ together with (i) of Lemma \ref{lemma:A} to
obtain $p_{h_i}h_{i+1}\le 2\a/(k-1)$. 
\smallskip

 \noindent
\subsubsection{Third step}[Recurrence].
With the above notation \eqref{PI} and \eqref{eq2.8} yield the following key recursive
inequality:
\begin{align*}
T_{\rm rel}(\bbT)\le  \l\left[2+ \frac{4\a  }{p(k-1)} \left(\sum^{n-1}_{i=0} \sqrt{T_{i}}\right)^2 \right]
\end{align*}
with $T_i$ given by \eqref{Ti} and $\l=2\frac{1-\d}{1-9\d}$. Suppose now that $L=\a^{N+1}$ and
$\ell = \a^{N}$ with $\a= (1-\d)^{-1}$. Then
$T_{\rm rel}(\bbT)=T_{N+1}$ and $n=N$. If we set $a_i:=\sqrt{T_i}$ then
we get
\[
a_{N+1}\le c \sum_{i=0}^N a_i,\quad c=\l^{1/2}\left(2+\frac{4\a}{p(k-1)} \right)^{1/2},
\]
which implies that $b_n:= \sum_{i=0}^n a_i$ satisfies $b_{N+1}\le
(1+c)b_N$. In conclusion 
\[
T_{\rm rel}(\bbT)= a_{N+1}^2\le b^2_{N+1} \le (1+c)^{2N}b^2_1.
\] 
%If we then let $\g_i:=\g(h_i)$ the
% above inequality reads
% \[
% \g_{N+1}\le C \left(\sum_{i=0}^N \sqrt{\g_i}\right)^2
% \]
% Assuming inductively that $T_i\le A h_i^\z =A\a^{\z i}$, we get that
% \begin{align*}
% T_{\rm rel}(\bbT)
% =
% T_{N+1} &\le  A \frac{2\lambda\a}{p(k-1)}h_{N+1}^\z \left[\sum_{i=0}^N \a^{-\frac{\z(N-i) }{2}}\right]^2 
% \\ &\le
% Ah_{N+1}^\z  \frac{2\a\lambda}{p(k-1)}
% \left[\frac{1}{\a^{\frac{\z}{2}}-1}\right]^2\\ &\le A h_{N+1}^\z\le A L^\z,
% \end{align*}
% provided that $\z$ is so large that (recall that $\d<1/9$ and $p=1/k$) 
% \[
%  \frac{2\lambda\a}{p(k-1)}\left(\frac{1}{\a^{\z/2}-1}\right)^2\leq 4\l \a \left(\frac{1}{\a^{\z/2}-1}\right)^2= \frac{8}{1-9\d}\left(\frac{1}{(1-\d)^{-\z/2}-1}\right)^2\le 1 .
% \]
The proof of the upper bound of $T_{\rm rel}(\bbT)$ in Theorem
\ref{th:j=k} is complete if the depth $L$ is of the form $\a^{n}, n\in
\bbN$. The extension to general values of $L$ follows at once from Lemma
\ref{gapmon}.

\subsection{Lower bound on the relaxation time $T_{\rm rel}$}
\label{lwb}
Let us consider as test function to be inserted into the
variational characterisation of $T_{\rm rel}(\bbT)$ the cardinality
$N_r$ of
the percolation cluster $\cC_r$ of occupied sites associated to the root $r$. More
formally 
\[
N_r(\eta):= \#\{x\in \bbT: \ \eta_y=1 \ \forall y\in \g_x\}
\] 
where $\g_x$ is the unique path in $\bbT$ joining $x$ to the root
$r$. Notice that $N_r$ can be written as $
N_r(\eta)= \eta_r\bigl(\sum_{i=1}^k N_{x_i}+1\bigr)$, 
where $\{x_i\}_{i=1}^k$ are the children of the root and $N_{x_i}$ denotes
the analogous of the quantity $N_r$ with $\bbT$ replaced by the sub-tree
$\bbT_{x_i}$ rooted at $x_i$. 

We now compute the variance and Dirichlet form of $N_r$. Clearly
\[c^{-1}\sum_{x\in \bbT}\mu(x \text{ is a leaf of } \cC_r)\le
\cD_{\bbT}(N_r)\le c \sum_{x\in \bbT}\mu(x \text{ is a leaf of }
\cC_r) \le c\mu(N_r)
\]  
for some constant $c=c(k)$.  Moreover  $\mu(N_r)=
p\left(k\mu(N_{x_1})+1\right)$ which, for $p=p_c=1/k$, implies  that $\mu(N_r)=(L+1)/k$. 
To compute $\var_{\bbT}(N_r)$ we use the above expression for $N_r$
together with the formula for conditional
variance to write
\begin{align}
\label{varf}
  \var_{\bbT}(N_r)&= \mu\left(\var_{\bbT}(N_r\tc \eta_r)\right) +
    \var_{\bbT}\bigl(\mu(N_r\tc \eta_r)\bigr)\nonumber\\
&= pk \var_{\bbT_{x_1}}(N_{x_1})  + \var_{\bbT}\bigl(\eta_r (k
  \mu(N_{x_1}) +1)\bigr)\\
&= \var_{\bbT_{x_1}}\bigl(N_{x_1}\bigr) +p(1-p)(L+1)^2.
\nonumber
\end{align}
Hence $ \var_{\bbT}(N_r)\geq c'L^3$ and 
%Hence $ \var_{\bbT}(N_r)=p(1-p)L^3$ and 
\[
T_{\rm rel}(\bbT)\ge \frac{\var_{\bbT}(N_r)}{\cD_{\bbT}(N_r)}\ge c'' L^2.
\]
\qed

% the indicator of the event
% ${E}:=\{\eta:B^{L}(\eta)_r=1\}$. Then
% $$\Var(f)=\mu(E)(1-\mu(E))=p_L(1-p_L)\geq p_L$$
% where we used the definition in Lemma \ref{lemma:A}.
%  Next we compute the Dirichlet form of $\cD(f)$. We first observe that if $x\in\bbT$ is not a leaf of $\bbT$ the corresponding contribution to the Dirichlet form vanishes.  If instead $x$ is a leaf of $\bbT$ so that $c_{\bbT,x}\equiv 1$ then
%  $$\mu_{\bbT}(\Var_x(f))=2\mu_{\bbT}(\eta\in E;\eta^x\not\in E)=2p^{L}\prod_{i=1}^L(1-p_i)^{k-1}$$
%  Thus, using point (iv) of Lemma \ref{lemma:A}, we have
% $$\cD_{\bbT}(f)=(kp)^L\prod_{i=1}^L(1-p_i)^{k-1}\leq c\,p_L^2.$$
 \section{The quasi-critical case: proof of Theorem \ref{th:2}}
Here we assume $p=p_c-\epsilon$, $\epsilon>0$ and, without loss of
generality, we assume that  $\epsilon k\ll 1$. Recall that we work
directly on the infinite tree $\bbT^k$. 
\subsection{Upper bound on the relaxation time $T_{\rm rel}$}
We first claim that, for any $\ell$ such that
$ 2\ell(1-\epsilon k)^{\ell}<1$,
one has
\begin{equation}\label{neweq}
\Var(f)\le \lambda \sum_{x\in\bbT^k}\mu\left(\Var_x(\mu_{\hat \bbT_x}(c_{\bbT,x}^{(\ell)}f)\right)
\end{equation}
with $\lambda=\frac{2}{1-2(\ell+1)(1-\epsilon k)^{\ell}}$.
The proof of \eqref{neweq} starts from inequality \eqref{D1}, whose
derivation does not depend on the value of $p$. 
After that we proceed as follows. Since $p=p_c-\epsilon$, 
Lemma \ref{lemma:A}(ii) implies that
\[
\mu_{T_x}(1-c_{\bbT,x}^{(\ell)})=\frac{p_{\ell}}{p}\leq (1-\epsilon
k)^{\ell}\quad \forall x\in \bbT^k.
\]
Thus
\begin{align*}
\var(f)&\le  \sum_{x\in\bbT^k} \!\mu\left[\Var_x(\mu_{\hat
    \bbT_x}(f))\right] \\
 &\le 2\sum_{x\in\bbT^k}\!\mu_\bbT\left[\Var_x(\mu_{\hat \bbT_x}(c_{\bbT,x}^{(\ell)}f)\right]
   +  2(\ell+1)(1-\epsilon k)^{\ell}\sum_{x\in \bbT^k}\mu[\Var_x(\mu_{\hat \bbT_x}(f)]
%32\frac{\delta}{1-\delta}\sum_{y\in\bbT} \mu_\bbT\left[\Var_y(\mu_{\hat \bbT_y}(f)\right] 
  \end{align*}
  and \eqref{neweq} follows.

Now choose $\ell=-2\frac{\log(\epsilon k)}{\epsilon k}$,
so that $\lambda<4$ in \eqref{neweq} for
any $\epsilon$ small enough, and define, for $x\in \bbT^k$, $\bbT_x$ as
the finite $k$-ary
tree rooted at $x$ of depth $\ell$. 

Exactly the same arguments leading to \eqref{basic}, but without the
subtleties of the intermediate scales $\{h_i\}_i$, show that
\begin{equation}
  \label{lampo}
  \mu\left(\Var_x(\mu_{\hat \bbT_x}(c_x^{(\ell)}f)\right) \le
  T_{\rm rel}(\bbT)\sum_{y\in\bbT_x} \mu\left(c_y
    \Var_y(f)\right).
\end{equation}
If we now combine \eqref{lampo} together with \eqref{neweq} we get
\begin{equation}\label{lunedi}
\Var(f)\le 4\ell\ T_{\rm rel}(\bbT) \cD(f)
\end{equation}
for all $\epsilon$ small enough.
Finally we claim that $T_{\rm rel}(\bbT)\le c \ell^\b$ for some
appropriate constants $c,\b$. 

To prove the claim it is
enough to observe that, in its proof for the case $p=p_c$ given in section
\ref{section:critical}, only \emph{upper bounds} on
percolation probabilities played a role. By monotonicity these bounds
hold for \emph{any} $p\le p_c$. Hence the claim.
In conclusion 
\[
\Var(f)\le c\ell^{1+\beta}\cD(f)
\]
and $T_{\rm rel}\le c \ell^{1+\b}= c' \epsilon^{-(1+\b)}$.

% on a tree of depth $\ell_c$ and empty leaves. 
% We proceed exactly as for the first step of the proof in the critical case  up to inequality \eqref{sabato} (see Section \ref{section:critical}). From this, setting $p=p_c-\epsilon$, inequality \eqref{PI} follows with a new constant $\lambda=2(1-\delta)(1-\epsilon k)/(1-9\delta-\epsilon k(1-\delta))$.
% We then proceed as in the second and third step of Section \ref{section:critical} and get
% \begin{equation}\label{martedi}
% T_{\rm rel}(\bbT_{\ell_c})\leq A\ell_c^{\beta}
% \end{equation}
% provided that $\beta$ satisfies
% $$\frac{4}{(1-9\delta-\epsilon k(1-\delta))}\frac{k}{(k-1)}\left(\frac{1}{\alpha^{\beta/2}-1}\right)^2\leq\frac{16}{(1-18\delta)}\left(\frac{1}{\alpha^{\beta/2}-1}\right)^2\leq 1$$
% where we choose $\epsilon$ sufficiently small so that $\epsilon k(1-\delta)<1/2$.
% Finally \eqref{lunedi} together with \eqref{martedi} yield the desired upper bound on $T_{\rm rel}(\bbT^k)$.
% %But on
% %such a tree the probabilities that a given vertex, say the root, does
% %not become empty after $n\le \ell_c-1$ iteration of the bootstrap map
% %decay in $n$ exactly as in the critical case $O(1/n)$. Therefore the
% %same proof for the critical case shows that there exists $\b>0$ such
% %that $\g(\ell_c)\le \ell_c^\beta$. The proof of the upper bound is
% %complete.
\subsection{Lower bound of the relaxation time $T_{\rm rel}$}
Thanks to Lemma \ref{gapmon}, $T_{\rm rel}\ge T_{\rm rel}(\bbT)$ for
any finite sub-tree $\bbT$.
We now choose $\bbT$ as the $k$-ary tree
rooted at $r$ with depth $\ell=\lfloor 1/\epsilon \rfloor$ and proceed exactly as in the proof of Theorem \ref{th:j=k}. 
Using the notation of section \ref{lwb} we have
\[
\cD_{\bbT}(N_r)\le c\mu(N_r)\le c' \ell
\] 
where we used the fact that the average of $N_r$ at $p<p_c$  is bounded from above
by the same average computed at $p=p_c$ since $N_r$ is increasing (w.r.t. the natural partial order in
$\O_{\bbT}$).
To compute $\var_{\bbT}(N_r)$ we proceed recursively starting from (cf \eqref{varf})
\begin{align*}
  \var_{\bbT}(N_r)
&= (1-k\epsilon)\var_{\bbT_{x_1}}(N_{x_1})  +
\frac{1-p}{p}\mu(N_r)^2\\
\mu(N_r)&= (1-\epsilon k)\mu(N_{x_1}) +p
\end{align*}
Since the number of steps of the iteration is $\lfloor 1/\epsilon
\rfloor$  one immediately concludes that $\mu(N_r)\ge c_k \ell$ and $
\var_{\bbT}(N_r)\ge c'_k \ell^3$ for
some constant $c_k$ depending only on $k$. 
Thus
\[
T_{\rm rel}\ge T_{\rm rel}(\bbT)\ge
\frac{\var_{\bbT}(N_r)}{\cD_{\bbT}(N_r)}\ge c \ell^2 = c\,\epsilon^{-2},
\]
for some constant $c>0$.
\section{Mixing times: proof of Theorem~\ref{lemmamixing}}
The specific statement (i) and (ii) are a direct consequence of
\eqref{eq:4}, Theorem \ref{th:j=k} and Theorem \ref{th:2}.
The upper bound $
T_1(\bbT)\leq T_2(\bbT)\leq cLT_{\rm rel}(\bbT)$
was proved in \cite {MT}*{Corollary 1]}. It remains to prove the
lower bound and this is what we do now following an idea of \cite{Ding}. 

Consider two probability measures $\pi,\nu$ on $\O_{\bbT}$ and recall their \emph{Hellinger} distance
\[
d_\cH(\pi,\nu):=\sqrt{2-2I_\cH(\pi,\nu)},
\]
where 
\[
I_\cH(\pi,\nu):= \sum_\o\sqrt{\pi(\o)\nu(\o)}.
\]
Clearly 
\[
I_\cH(\pi,\nu)\ge \sum_{\eta\in \O_\bbT}\pi(\eta)\wedge \nu(\eta)\ge
1- \|\pi-\nu\|_{TV}.
\]
If we combine the above inequality with \cite{Evans}*{Lemma 4.2 (i)}
we get 
\begin{align*}
 \frac 12 d_\cH(\pi,\nu)^2\le  \|\pi-\nu\|_{TV}\le d_\cH(\pi,\nu).
\end{align*}
Assume now that $\pi,\nu$ are product measures, $\pi=\prod_{i=1}^n
\pi_i,\ \nu=\prod_{i=1}^n\nu_i$, so that
\[
I_\cH(\pi,\nu):= \prod_{i=1}^n I_\cH(\pi_i,\nu_i).
\] 
Therefore
\begin{align}
\label{eq:1e}  
\|\pi-\nu\|_{TV}&\ge 1-I_\cH(\pi,\nu)=1-\prod_{i=1}^n I_\cH(\pi_i,\nu_i)\nonumber\\
&=1- \prod_{i=1}^n \left(1-\frac 12 d_\cH(\pi_i,\nu_i)^2\right) \nonumber\\
&\ge 1-\prod_{i=1}^n \left(1-\frac 12 \|\pi_i-\nu_i\|^2_{TV}\right) \nonumber\\
&\ge 1-e^{-\sum_i \frac 12 \|\pi_i-\nu_i\|^2_{TV}}.
\end{align}  
Suppose now that, for each $i\le n$, $\nu_i$ is the distribution at time $t$ of some finite, ergodic, continuous time Markov chain $X^{(i)}$, reversible w.r.t. $\pi_i$  and with initial state $x_i$. 
In this case the measure $\nu$ is the distribution at time $t$ of the
product chain $X=\otimes_i X_i$ started from $x=(x_1,\dots,x_n)$ and $\pi$ is the reversible measure .

Let $\l_i$ be the spectral gap of the chain $X^{(i)}$, let $f_i$ be the corresponding eigenvector  and choose the starting state $x_i$ in such a way that
$|f_i(x_i)|=\|f_i\|_\infty$. Then
\begin{align}
\label{eq:2e}
\|\pi_i-\nu_i\|_{TV}&\ge \frac 12 \frac {1}{\|f_i\|_\infty} |\pi_i(f_i)-\nu_i(f_i)|=\frac 12 \frac{|f(x_i)|}{\|f_i\|_\infty}e^{-\l_i t}\nonumber\\
&=\frac 12 e^{-\l_i t},
  \end{align}
where we used $\pi_i(f_i)=0$ because $f_i$ is orthogonal to the constant functions.

In conclusion, by combining together \eqref{eq:1e} and \eqref{eq:2e}, we get
\[
\|\pi-\nu\|_{TV}\ge 1-e^{-\frac 18\sum_i e^{-2\l_i t}}.
\]
Therefore, if $t=t^*$ with
\[
t^*=\frac 12\left[\frac 1{\max_i \l_i}\log n - \frac{1}{\min_i \l_i}\log 8\right],
\] 
then $\|\pi-\nu\|_{TV}\ge 1-e^{-1 }$.
Thus the mixing time of the product chain $X$ is larger than $t^*$. \\
We now apply the above strategy to prove a lower bound on
$T_1(\bbT)$.

Let $\bbT^{(i)}$ be the $i^{th}$ (according to some arbitrary order)
$k$-ary sub-tree of depth $\lceil L/2\rceil$ rooted at the
$\lfloor L/2\rfloor $-level of $\bbT$ and consider the OFA-kf model on
$\cup_i \bbT^{(i)}$. Clearly such a chain $X$ is a product chain, $X=\otimes_i X_i$, where each of the individual chain is the OFA-kf model on $\bbT^{(i)}$. The key observation now is that, due to the oriented character of the constraints, the projection on $\cup_i \bbT^{(i)}$ of the OFA-kf model on $\bbT$ coincides with the chain $X$. Hence $T_1(\bbT)\ge t_{\rm mix}$
if $t_{\rm mix}$ denotes the mixing time of the product chain
$X$. According to the previous discussion and with $n=k^{\lfloor
  L/2\rfloor}$ the number of sub-trees $T^{(i)}$ we get
\begin{align*}
T_1(\bbT)\ge t_{\rm mix} &\ge \frac 12 \left(\log n - \log 8\right) \gap(\cL_{\bbT'})^{-1}=
\frac 12\left(\log n - \log 8\right)T_{\rm rel}(\bbT')\\
&\ge \frac 1c L\, T_{\rm rel}(\bbT')
\end{align*}
for some constant $c>0$ where we used translation invariance to conclude  that the spectral gap
$\l_i$ of the chain $X_i$ coincides with $\gap(\cL_{\bbT'})$ for any
$i$, $\bbT'$ denoting a $k$-ary rooted tree of depth $\lceil L/2\rceil$.

\section{Concluding remarks and open problems}
({\bf i}) It is a very interesting problem to determine exactly the critical
exponents for the critical and quasi-critical case and in particular 
to verify whether the \emph{lower bounds} in Theorems \ref{th:j=k} and
\ref{lemmamixing} give the correct growth of the corresponding time
scales as a function of the depth of the tree. 
\smallskip

({\bf ii}) A key ingredient of our analysis
is the fact that the percolation transition on $\bbT^k$ is
\emph{continuous}, i.e. with probability one there is no infinite
cluster of occupied sites at $p=p_c$ and the probability that the
cluster of the root touches more than $n$ levels decays polynomially
in $1/n$. A very challenging open problem is the extension of the
approach described in this work to models with a \emph{discontinuous} (or first-order) phase transition for the corresponding
bootstrap percolation problem. 

The first instance of the above general question goes
as follows. On $\bbT^3$ consider the analog of the
OFA-kf model in which
the constraint at each vertex $x$ requires now at least two of the three children
of $x$ to be
empty. It can be shown \cite{Boot-Peres} that the critical value 
of the corresponding bootstrap percolation problem is $p_c=\frac 89$
and that after infinitely many iterations of the bootstrap map the
root belongs to an infinite cluster of occupied sites with probability
equal to $\frac 34$. In \cite{MT} it was proved that $T_{\rm
  rel}<+\infty$ for all $p<p_c$. At $p_c$ the process is clearly no
longer ergodic, contrary to what happens for the OFA-kf model, because
of the presence of infinite bootstrap percolation clusters which  are
blocked under the dynamics. Finally, for
$p>p_c$, the relaxation time on a finite sub-tree diverges
exponentially fast in the depth of the tree. 

The interesting challenge is to decide the behaviour of e.g. the
relaxation time on a finite $3$-ary rooted tree of finite depth
$L$ at $p_c$. On one hand, the
fact that 
\[
\bbP(\text{the root belongs to a occupied
cluster reaching the leaves}) \sim 3/4,
\] may suggest a scaling of $T_{\rm rel}$ in $L$ much more
rapid than for the critical OFA-kf, even faster than $\text{Poly}(L)$. On the
other hand, the test function given by the indicator of the event
that the root is still occupied after $L$ iterations of the bootstrap
map, which at $p>p_c$ gives an exponential growth in $L$ of $T_{\rm
  rel}$, at $p=p_c$ gives $T_{\rm rel}=\O(L^2)$, exactly as in the
OFA-k model. The same bound $\O(L^2)$ is found using another test function
closer to the one used in section \ref{lwb}. 

Here we conjecture that $T_{\rm rel}$ is still $\text{Poly}(L)$. This conjecture is supported by  numerical
simulations for the unoriented version of the same model \cite{SBT},
namely the model on the unrooted tree with connectivity $k+1=4$ in
which the kinetic constraint requires at least two empty neighbours
(actually these numerical results concern the relaxation time of the
\emph{persistence function} in the quasi-critical regime, a new time scale
which can be bounded from above by $T_{\rm rel}$ \cite{CMRT}). Another
element in favour of our guess is the fact that the phase transition
occurring at $p_c$ has really a mixed first-second order character as
indicated by some non-rigorous work \cite{CLR,Goltsev}. 

% \bibliographystyle{alpha}
 %\bibliography{critical-tree}

\begin{bibdiv}
\begin{biblist}


   






\bib{AD}{article}{
      author={Aldous, D.},
      author={Diaconis, P.},
       title={The asymmetric one-dimensional constrained {I}sing model:
  rigorous results},
        date={2002},
     journal={J. Statist. Phys.},
      volume={107},
      number={5-6},
       pages={945\ndash 975},
}

%\bib {BS}{article}{
%author={Biskup, M.},
%      author={Schonmann, R.},
%        TITLE = {Metastable behavior for bootstrap percolation on regular trees},
%   JOURNAL = {J.Stat.Phys.},
%  FJOURNAL = {Journal of Statistical Physics},
%    VOLUME = {136},
%      YEAR = {2009},
%     PAGES = {667-676.},
%}

\bib{Boot-Peres}{article}{
      author={Balogh, J{\'o}zsef},
      author={Peres, Yuval},
      author={Pete, G{\'a}bor},
       title={Bootstrap percolation on infinite trees and non-amenable groups},
        date={2006},
     journal={Combin. Probab. Comput.},
      volume={15},
      number={5},
       pages={715\ndash 730},
}

\bib{CMRT}{article}{
      author={Cancrini, N.},
      author={Martinelli, F.},
      author={Roberto, C.},
      author={Toninelli, C.},
       title={Kinetically constrained spin models},
        date={2008},
     journal={Probability Theory and Related Fields},
      volume={140},
      number={3-4},
       pages={459\ndash 504},
}

\bib{CMST}{article}{
      author={Cancrini, N.},
      author={Martinelli, F.},
      author={Schonmann, R.},
      author={Toninelli, C.},
       title={Facilitated oriented spin models: some non equilibrium results},
        date={2010},
     journal={J. Stat. Phys.},
      volume={138},
      number={6},
       pages={1109\ndash 1123},
}

\bib {CLR}{article}{
 author={Chalupa,J },
      author={Leath, P.L.},
      author={Reich, G.R.},
        TITLE = {Bootstrap percolation on a Bethe lattice},
   JOURNAL = {J.Phys.C},
  FJOURNAL = {Journal of Physics C},
    VOLUME = {12},
      YEAR = {1979},
     PAGES = {L31-L37},
}

\bib{Ding}{article}{
      author={Ding, Jian},
      author={Lubetzky, Eyal},
      author={Peres, Yuval},
       title={Mixing time of critical {I}sing model on trees is polynomial in
  the height},
        date={2010},
     journal={Comm. Math. Phys.},
      volume={295},
      number={1},
       pages={161\ndash 207},
}

\bib{Evans}{article}{
      author={Evans, William},
      author={Kenyon, Claire},
      author={Peres, Yuval},
      author={Schulman, Leonard~J.},
       title={Broadcasting on trees and the {I}sing model},
        date={2000},
     journal={Ann. Appl. Probab.},
      volume={10},
      number={2},
       pages={410\ndash 433},
}

\bib{FMRT}{article}{
      author={Faggionato, A.},
      author={Martinelli, F.},
      author={Roberto, C.},
      author={Toninelli, C.},
       title={Aging through hierarchical coalescence in the east model},
        date={2012},
     journal={Com. Math. Phys.},
      volume={309},
       pages={459\ndash 495},
}

\bib{FMRT3}{unpublished}{
      author={Faggionato, A.},
      author={Martinelli, F.},
      author={Roberto, C.},
      author={Toninelli, C.},
       title={The {E}ast model: recent results and new progresses},
        date={2012},
        note={preprint available at arXiv:1205.1607},
}

\bib{GarrahanSollichToninelli}{incollection}{
      author={Garrahan, J.},
      author={Sollich, P.},
      author={Toninelli, C.},
       title={Dynamical heterogeneities and kinetically constrained models},
        date={2011},
   booktitle={Dynamical heterogeneities in glasses, colloids and granular media
  and jamming transitions},
   publisher={Oxford press,Ed: L. Berthier,G. Biroli J-P. Bouchaud,
  L.Cipelletti , W.van Saarloos},
     address={Oxford},
       pages={341\ndash 369},
}

\bib{Goltsev}{article}{
      author={Goltsev, A.~V.},
      author={Dorogovtsev, S.~N.},
      author={Mendes, J. F.~F.},
       title={$k$-core (bootstrap) percolation on complex networks: Critical
  phenomena and nonlocal effects},
        date={2006},
     journal={Phys. Rev. E},
      volume={73},
       pages={056101},
}

\bib{JACKLE}{article}{
      author={J\"{a}ckle, J.},
      author={Eisinger, S.},
       title={A hierarchically constrained kinetic ising model},
        date={1991},
     journal={Z. Phys. B: Condens. Matter},
      volume={84},
      number={1},
       pages={115\ndash 124},
}
\bib{LK}{article}{
    author = {Kordzakhia, G.}
     author={Lalley, S.P.},
     TITLE = {Ergodicity and mixing properties of the northeast model},
   JOURNAL = {J. Appl. Probab.},
  FJOURNAL = {Journal of Applied Probability},
    VOLUME = {43},
      YEAR = {2006},
    NUMBER = {3},
     PAGES = {782-792.},
}

\bib{Liggett}{book}{
      author={Liggett, T.M.},
       title={Interacting particle systems},
      series={Grundlehren der Mathematischen Wissenschaften [Fundamental
  Principles of Mathematical Sciences]},
   publisher={Springer-Verlag},
     address={New York},
        date={1985},
      volume={276},
}

\bib{Peresbook}{book}{
    author = {Levin, David A.},
    author = {Peres, Yuval},
    author = {Wilmer, Elizabeth},
    title =  {Markov Chains and Mixing Times},
    publisher = {American Mathematical Society},
    date = {2008},
    pages = {371},
}



\bib{MT}{unpublished}{
      author={Martinelli, F.},
      author={Toninelli, C.},
       title={Kinetically constrained spin models on trees},
        date={2012},
        note={Annals of Applied Probability, in press},
}

\bib{Ritort-Sollich}{article}{
      author={Ritort, F.},
      author={Sollich, P.},
       title={Glassy dynamics of kinetically constrained models},
        date={2003},
     journal={Adv Phys},
      volume={52},
      number={4},
       pages={219\ndash 342},
}

\bib{Gine}{incollection}{
      author={Saloff-Coste, L.},
       title={Lectures on finite {M}arkov chains},
        date={1997},
   booktitle={Lectures on probability theory and statistics ({S}aint-{F}lour,
  1996)},
      series={Lecture Notes in Math.},
      volume={1665},
   publisher={Springer},
     address={Berlin},
       pages={301\ndash 413},
}
\bib{SBT}{article}{
author={Sellitto, S.},
      author={Biroli, G.},
author={Toninelli, C.},
     TITLE = {Facilitated spin models on Bethe lattice: bootstrap percolation, mode coupling transition and glassy dynamics},
   JOURNAL = {Europhysics Lett.},
  FJOURNAL = {Europhysics Letters},
    VOLUME = {69},
      YEAR = {2005},
     PAGES = {496-512.},
}
\end{biblist}
\end{bibdiv}

%

\end{document}